\documentclass[11pt]{article}
\usepackage[utf8]{inputenc}
\usepackage{amsfonts, amsmath}
\usepackage{amssymb}
\usepackage{amsthm,amsmath,amscd}
\usepackage{enumerate}
\usepackage{url}
\usepackage[pdftex]{graphicx}
\usepackage[margin=25pt,font=small,labelfont=bf]{caption}
\usepackage{subfig}
\usepackage[numbers,square,sort&compress]{natbib}
\usepackage[colorlinks=true]{hyperref}

\usepackage[
    colorinlistoftodos,
    backgroundcolor=yellow,
    textsize=footnotesize,
]{todonotes}

\newcommand{\defn}[1]{\textcolor{blue}{\emph{#1}}}
\newcommand*{\doi}[1]{doi: \href{https://dx.doi.org/#1}{\urlstyle{rm}\nolinkurl{#1}}}
\newcommand*{\arxiv}[1]{arXiv:  \href{https://arxiv.org/abs/#1}{\urlstyle{rm}\nolinkurl{#1}}}
\let\oldproofname=\proofname
\renewcommand{\proofname}{\rm\bf{\oldproofname}}

\newcommand{\RR}{\mathbb R}

\newcommand{\bna}{\begin{eqnarray}}
\newcommand{\ena}{\end{eqnarray}}
\newcommand{\ba}{\begin{eqnarray*}}
\newcommand{\ea}{\end{eqnarray*}}
\newcommand{\bs}[1]{}

\newtheorem{theorem}{Theorem}[section]

\newtheorem{lemma}[theorem]{Lemma}

\newtheorem{definition}[theorem]{Definition}
\newtheorem{conjecture}[theorem]{Conjecture}

\DeclareMathOperator{\tr}{tr}

\def\p{{\bf p}}
\def\hp{\hat{\bf p}}

\def\hx{\hat{\bf x}}

\def\e{{\bf e}}

\def\q{{\bf q}}
\def\0{{\bf 0}}

\let\oldv=\v
\def\v{{\bf v}}
\def\t{{\bf t}}

\def\e{{\bf e}}

\def\one{{\bf 1}}

\usepackage[margin=1in]{geometry}

\begin{document}
\title{The Stress-Flex Conjecture}

\author{
Robert Connelly
\and
Steven J. Gortler
\and
Louis Theran
\and
Martin Winter}
\date{}
\maketitle

\begin{abstract}
Recently, it has been proven that a tensegrity
framework that arises from coning the one-skeleton of
a convex polytope is rigid.
Since such frameworks are not always infinitesimally rigid, this leaves open 
the question as 
to whether they are at least prestress stable. We prove here
that this holds subject to an intriguing new conjecture about coned polytope frameworks, that we call the \textit{stress-flex conjecture}. Multiple numerical experiments suggest that this 
conjecture is true, and most surprisingly, seems to hold even beyond convexity and also for higher genus~polytopes.
\end{abstract}

\section{Introduction}
It was recently proven by Winter~\cite{winter} that 
if one takes the one-skeleton of a convex
polytope in $\RR^d$ and  cones it over an interior point, the resulting framework
is rigid. It is rigid as a bar-framework, where the edge lengths must stay
constant, and it is even rigid  as a tensegrity framework, where the ``polytope edges'' are allowed 
to get shorter and the ``cone edges'' are allowed to get longer. 

As an example, consider the graph
$G$ with $8$ vertices and $12$ edges
obtained from the one-skeleton of a cube 
in $\RR^3$.
Cone this graph over a $9$th point
to obtain the graph $G^*$. If we forget about the
cube's geometry and just choose
a generic placement $\hp$ of the $9$ vertices
in $\RR^3$, 
the pair 
$(G^*,\hp)$,
considered as a bar framework, will be flexible;
it will have 
a non-trivial infinitesimal flex and 
no (non-zero) equilibrium stresses. 
Let us now consider the (non-generic) frameworks
$(G^*,\hp)$,
where we use the original geometry of the cube
for the first $8$ vertices and place the the
cone point anywhere in $\RR^3$.
Such a
bar framework will have 
a one dimensional equilibrium stress space and
a two dimensional space of infinitesimal flexes
complementary to the trivial flexes.
The actual   rigidity 
of such a  bar-framework  appears to be
a subtle phenomenon and depends on exactly where the cone
point is placed.
For some placements of the cone point \emph{outside} of the cube,
we have done numerical experiments that suggest that the bar framework
will be flexible. 
For some other such placements outside of the cube, we find that the
framework is (certifiably) rigid. 
In this context,
the result of~\cite{winter} tells us that 
whenever the cone point is 
\emph{inside} the cube, the framework
must be rigid, both as a bar framework and even as a 
(less constrained) tensegrity framework.

In the rigidity literature~\cite{pss} there
is a condition called prestress stability,
which is stronger than rigidity
(though weaker than infinitesimal rigidity).
Prestress stability is the main
property used to analyze frameworks that
are not infinitesimally rigid but that might be rigid.
The property can be efficiently tested by 
solving an appropriate  semi-definite
program.

After establishing rigidity for
interior coned polytopes, in \cite[Question 5.2.]{winter} the paper asks whether they are also always prestress stable. 
Establishing prestress stability would place the
rigidity of coned polytope frameworks more clearly
within the understood rigidity landscape.
There are also practical implications. When a framework
is prestress stable, one can also use eigen-analysis to
obtain certain bounds on the robustness of this rigidity
to slight violations of the length constraints of the tensegrity
and slight alterations to the original configuration (see e.g., ~\cite[Corollary 3.3]{almost}).

In this note we explore this question, and
show that it will be answered
in the affirmative if one is able to 
prove a conjecture we pose below, that we call
the \emph{stress-flex conjecture}. This
perhaps odd-looking conjecture
is of intrinsic interest and appears to hold under numerous numerical experiments in
three dimensions, even beyond the scope needed for prestress stability.
So we are fairly hopeful that the stress-flex conjecture, and thereby prestress stability, can be verified.

In terms of techniques, \cite{winter}
relies on a result of Izmestiev~\cite{izCdv}
that generalizes a three-dimensional result of 
Lov{\'a}sz~\cite{lovS}. This result associates a special
$n$-by-$n$ 
matrix $M$ to a convex polytope 
$P$
(with $n$ vertices) in $\RR^d$ that
contains the origin in its interior. Among other things,
$M$ has a single negative eigenvalue.
In this note, we add an extra row and column to $M$
to obtain $(n+1)$-by-$(n+1)$ 
matrix $\Omega$ with a single negative
eigenvalue.
By construction, this matrix will be a stress matrix
for the framework of the coned polytope. 
This stress matrix is then used in our exploration of prestress stability.

\subsection*{Acknowledgements}
We thank Dylan Thurston
for help simplifying Lemma~\ref{lem:psd}.
We thank Miranda Holmes-Cerfon for the use of
her numerical rigidity-simulation code.

\section{Preliminaries}
Here we quickly review the basic rigidity definitions that
we will need. The key reference is~\cite{pss}.

\begin{definition}
A \defn{configuration} $\p$ of $n$ points 
in $\RR^d$ is an ordered set
of $n$ points, $\p_i \in \RR^d$.
\end{definition}

\begin{definition}
A \defn{tensegrity framework} $(G,\p)$ in $\RR^d$ is 
a configuration $\p$ of $n$ points $\RR^d$, and a 
labeled graph $G$ on $n$ vertices. 
The edges of $G$ 
are labeled as ``cables'', ``struts'' and ``bars''.
We denote by $(\bar{G},\p)$ the associated 
\defn{bar framework}, where all of the labels in $G$
are changed to~be ``bars''.
\end{definition}

\begin{definition}
A tensegrity framework $(G,\p)$ is (locally) \defn{rigid}
if for every configuration $\q$ of $n$ points in $\RR^d$
that is sufficiently close to $\p$\footnote{Formally, there exists an $\epsilon$ so that 
for every $\q$ within this distance to $\p$, the property holds.} 
and such that $\q$ is
not congruent to $\p$, 
we have the property that the Euclidean lengths of the cables
are not decreased, the lengths of the struts are not increased
and the lengths of the bars are not changed.
\end{definition}

\begin{definition}
Let $(\bar{G},\p)$ be a bar framework in $\RR^d$.
We say that a configuration $\p'$ of $n$ vectors in $\RR^d$
is an \defn{infinitesimal flex} for the framework if for all
edges $ij$, we have $(\p_i-\p_j)\cdot(\p'_i-\p'_j)=0$.

An infinitesimal flex is \defn{trivial} if there is a $d$-by-$d$
skew symmetric matrix $S$ and a vector $\t$ such that for all 
$i$,
$\p'_i=S\p_i+\t$. (Such a flex arises as
the 
time derivative of a continuous isometry of $\RR^d$.)
\end{definition}

If the only infinitesimal flexes for $(\bar{G},\hat{\p})$ are trivial ones, then
the bar framework is 
said to be \defn{infinitesimally rigid} which
implies that it must be
rigid.  The converse is not 
necessarily true.
(There is also an infinitesimal notion
in the context of tensegrity frameworks, but we
will not go into this.)

Next we develop the idea of prestress stability.
The rough idea is to establish rigidity by 
defining a certain energy function over $\p$
that only depends on the edge lengths,
and then show that $\p$ is a critical point 
such that
the energy has a positive definite Hessian.
When all of the dust settles, we are left with 
the following definitions and theorem.

\begin{definition}
Let $(G,\p)$ be a tensegrity with $n$ vertices.
An (equilibrium) \defn{stress matrix} $\Omega$ of $(G,\p)$
is an $n$-by-$n$
symmetric matrix with the following properties:
\begin{enumerate}[(i)]
    \item $\Omega_{ij}=0$ when $i\ne j$ and $ij$ is not an edge of $G$.
    \item 
    $\Omega \p=0$, where we think of $\p$ as an $n$-by-$d$ matrix.
    \item 
    $\Omega \one=0 $, where $\one$ is a vector of all ones.
\end{enumerate}
The stress is called \defn{strictly proper} if $\Omega_{ij} < 0$
when $ij$ is a cable and $\Omega_{ij}>0$ when $ij$
is a strut.
\end{definition}

\begin{definition}[{See~\cite[Proposition 3.4.2]{pss}.}]
A tensegrity framework $(G,\p)$
in $\RR^d$
is \defn{prestress stable} 
if it has a 
strictly proper equilibrium stress
matrix $\Omega$, such that  for every 
non-trivial infinitesimal flex $\p'$
of its bar-framework $(\bar{G},\p)$,
we have 
\bna
\label{eq:pss}
\tr(\p'^t \Omega \p') > 0
\ena
where we think of $\p'$ as an $n$-by-$d$ matrix.
\end{definition}

Finally, we recall the following.
\begin{theorem}[{\cite[Proposition 3.3.2]{pss}}]
If a tensegrity framework 
is prestress stable, then it is rigid.  
\end{theorem}
The converse does not always hold.

\section{Izmestiev Stress}
\label{sec:iz}
Let $P\subset\RR^d$ be a convex polytope with 
$n$ vertices, with a
full affine span and with the origin
in its interior.
Let $\p$ be its vertex configuration and let $G$ be 
the graph of its one-skeleton. 
Izmestiev~\cite{izCdv} constructs 
an $n$-by-$n$ symmetric matrix $M$, that we call the
\defn{Izmestiev matrix}, with the following properties:
\begin{enumerate}[(i)]
    \item 
$M_{ij}=0$ when $i\neq j$ and $ij$ is not an edge of $G$.
\item $M_{ij} <0$ when $ij$ is an edge of $G$.
 \item $M\p=0$, where we think of $\p$ as an $n$-by-$d$
matrix. 
\item $M$ has rank $n-d$. 
\item $M$ has exactly one negative eigenvalue.
\end{enumerate}

Our goal is to use $M$ to construct an equilibrium stress
$\Omega$ for the framework $(G^*,\hp)$ obtained by coning the one-skeleton of $P$ over the origin.
Let $\alpha^t := -\one^t M$ and 
$b:=-\sum_i \alpha_i = \one^tM\one$.
The $\alpha_i$ are also known as the 
(unnormalized)
Wachspress coordinates
of the origin with respect to
$P$~\cite{winter}, 
and so each is greater than $0$, and $b$ is negative.
We define the $(n+1)$-by-$(n+1)$ matrix in block form as follows:
\ba
\Omega :=
\begin{pmatrix}
M &\alpha \\
\alpha^t & b
\end{pmatrix}
\ea
The added row/column are linearly dependent on  $M$,
and thus the rank of $\Omega$ is also $n-d$, and its nullity
is $d+1$. 
Since $M$ has one negative eigenvalue, and $\Omega$
just has an extra $0$ eigenvalue, 
from the eigenvalue interlacing theorem, $\Omega$ must
also have exactly one negative eigenvalue.

Let 
$\p_{n+1}$ be placed at the origin
and let $\hp:=[\p,\p_{n+1}]$ be the configuration 
of $n+1$ points in $\RR^d$.
Then, since $\p_{n+1}=0$ we have $\Omega \hp=0$.
By our definition of $\alpha$ and $b$, 
we have $\Omega \one=0$.
Let us label the edges of $G^*$
from the polytope as cables and the
coned edges as struts.
We see that  $\Omega$ is a
strictly proper equilibrium stress for $(G^*,\hp)$.
We call this an \defn{Izmestiev stress}.

Finally we note that the space of stresses for a framework
does not change under translation in $\RR^d$.
Thus we can start with $P$ as any convex polytope with
$n$ vertices and a full affine span in $\RR^d$ and with
$\p_{n+1}$ any point in the interior of $P$. We can create
its associated coned tensegrity, and it must have an 
Izmestiev stress.

\section{One Negative Eigenvalue}

Prestress stability, and in particular Equation (\ref{eq:pss}),
is based on positivity of a quadratic energy. Our stress matrix
$\Omega$ has one negative eigenvalue that we need to 
work around. Our main tool for doing this is the following
lemma.

\begin{lemma}
\label{lem:psd}
Let $\Omega$ be an 
$(n+1)$-by-$(n+1)$ real symmetric
matrix. Assume the following:
\begin{enumerate}[(i)]
\item $\Omega$ has exactly one negative eigenvalue.
\item $\Omega_{n+1,n+1}<0$.
\end{enumerate}
If $\hx \in \RR^{n+1}$
has the property that the $(n+1)$st
 entry of $\Omega\hx$ equals
$0$,
then $\hx^t \Omega \hx\ge 0$.
\end{lemma}
\begin{proof}
Let $\e$ be the $(n+1)$st indicator vector.
By assumption $\e^t\Omega\e<0$, so the vector, $\e$ is~in the negative 
cone of $\Omega$. 
Meanwhile, since the 
last entry of $\Omega\hx$ equals $0$,
we have 
$\e^t\Omega\hx=0$,
i.e., $\hx$ is $\Omega$-orthogonal to $\e$.
Together, since $\Omega$ has only a single negative eigenvalue,  this implies $\hx^t\Omega\hx\ge 0$.
\end{proof}

Similar ideas to Lemma~\ref{lem:psd} are used
in~\cite{izInf}.

\section{The Stress-Flex Conjecture}

The Izmestiev stress $\Omega$ satisfies the assumptions
of Lemma~\ref{lem:psd}. If we can show that each of
the $d$ spatial coordinates of $\hp'$ have the
properties of the vector $\hx$ in that lemma, then we can apply this to establish prestress stability.
The following conjectures just that:

\begin{conjecture}[The \defn{stress-flex conjecture}]
\label{conj:stress_flex}
Let $P$ be  a convex polytope 
with $n$ vertices with a full dimensional span in $\RR^d$.
Let $(G^*,\hp)$ be the tensegrity framework
on $n+1$ vertices 
defined by putting cables on the 
one-skeleton of $P$ and struts from 
the vertices of $P$ to a selected 
cone point in the interior of $P$.
Let $\Omega$ be its
Izemstiev stress. Let $\hp'$ be an infinitesimal flex
of its bar framework.
Then 
the last row of 
$\Omega \hp'$ equals zero.
\end{conjecture}

When the last row of 
$\Omega \hp'$ equals zero, we say that the
\defn{stress-flex condition} is satisfied.
Note that this condition is always satisfied for
trivial infinitesimal flexes, since in this case
$\hp'$ is an 
affine image of $\hp$ and so $\Omega\hp'=0$.

We have done many numerical experiments in 3D whose
results support this conjecture. These results
included (simple) convex polytopes defined by the intersection of 
random planes as well a number of polytopes from the
wonderful library
at~\cite{polys}. 

In fact, this appears to
be an even more robust property than stated by the conjecture.
First of all, the stress-flex condition
experimentally holds even when 
we place the cone point outside of the polytope.
We have also looked at polytopes where the equilibrium stress space of $(G^*,\hp)$ has
dimension larger than one. 
For example,
 when $P$ is the cuboctahedron, the coned framework has
a $4$ dimensional stress space 
and it has a non-trival infinitesimal flex. When 
$P$ is  the 
rhombic dodecahedron, the coned framework has a $2$ dimensional stress space and it has a  
$3$ dimensional space of flexes complementary
to the trivial flexes.  In these cases, the stress-flex condition
holds for any of the stresses, not just the Izmestiev stress.

We have also looked at non-convex polyhedra and non-embedded
polyhedra where there are stresses and non-trivial
infinitesimal flexes,
and the stress-flex condition
continues to
hold. We have looked at higher genus polyhedra, such as
Szilassi Polyhedron.
We have even looked at non-orientable polyhedra such as the 
cubohemioctahedron. In all of these cases, the stress-flex condition still holds.
Moreover, the stress-flex condition holds for the hypercube in $\RR^4$,
coned anywhere.

In light of these experiments, we propose the 
following \defn{strong stress-flex conjecture}. 

\begin{conjecture}
Let $P$ be any ``polytope'' 
with a full dimensional affine span in $\RR^d$.
Let $\p_{n+1}$ be placed anywhere in $\RR^d$.
Let $\Omega$ be \emph{any}
stress for the bar framework defined using
one-skeleton of $P$ coned over $\p_{n+1}$. Let $\hp'$ be any infinitesimal flex of the bar framework.
Then 
the the stress-flex condition holds.
\end{conjecture}
In this conjecture, we have chosen to not exactly pin-down
what we mean by ``polytope'', but it seems that we need
nothing more than a closed surface with flat faces.

We  note that if one starts with a convex polytope coned 
from inside, and then slides the vertices along their lines to the cone vertex, 
even though the face-flatness is not maintained,
the dimension of the
stress space and infinitesimal flex space is 
maintained.
It also can be shown that the property of prestress stability is maintained under such sliding.
But we observe that for generic sliding the stress-flex 
condition
is not maintained. So the face-flatness appears to be
crucial for the stress-flex condition.

\subsection{Another Point of View}
It can also be shown that the strong stress-flex conjecture is equivalent
to the following:
Start with any polytope $P$ 
with full affine span 
in $\RR^d$. 
Let $(G,\p)$ be 
the bar framework of its one-skeleton.
Let $(G,\q)$ be the orthogonal projection to $\RR^{d-1}$ by forgetting the last
coordinate.
Let $\Psi$ be any equilibrium stress matrix for $(G,\q)$.
Let $\q'$ be any infinitesimal flex of $(G,\q)$ in $\RR^{d-1}$.
Then the following two conditions must hold:
\ba
(\p^d)^t \Psi \q'= 0
\ea
where $\p^d$ is the $d$th spatial coordinate of $\p$ and we think of $\q'$ as an $n$-by-$(d-1)$ matrix
and 
\ba
(\p^d)^t \Psi (\q'\cdot \q)= 0
\ea
where $(\q'\cdot \q)_i := \q'_i \cdot \q_i$.

\section{Prestress Stability of a 
Convex Polytope, Coned 
From Inside}

Finally we tie things together, and prove that 
under the (weak) stress-flex conjecture, we must have
prestress stability. 
There is not much to do here,
other than use the properties
of the Izmestiev stress and Lemma~\ref{lem:psd}.
This will tell us that any infinitesimal flex must 
be an affine flex. All that is left is to string 
together some previous results that let us conclude
that this affine flex must be a trivial flex.

\begin{theorem}
Let $P$ be  a convex polytope 
with $n$ vertices with a full dimensional span in $\RR^d$.
Let $(G^*,\hp)$ be the tensegrity framework
on $n+1$ vertices 
defined by putting cables on the 
one-skeleton of $P$ and struts from 
the vertices of $P$ to a selected 
cone point in the interior of $P$.

Assume the (weak) stress-flex conjecture (Conjecture~\ref{conj:stress_flex}) is true,
then the 
resulting framework is prestress stable.
\end{theorem}
\begin{proof}
From Section~\ref{sec:iz}, $(G^*,\hp)$ must have
an Izmestiev stress. This stress has exactly one negative
eigenvalue and its lower-right entry is negative.
Under the stress-flex conjecture, for any 
infinitesimal flex $\hp'$ of
of $(\bar{G},\hp)$, we have the last row of 
$\Omega \hp'$ is equal to zero. 
Using Lemma~\ref{lem:psd} we must have
$\tr(\p'^t \Omega \p') \ge 0$. 

Next we show that if
$\tr(\p'^t \Omega \p') =0$, then $\p'$ must be a trivial
flex.  From Lemma~\ref{lem:psd}, if
$\tr(\p'^t \Omega \p') =0$ then we must have   
(each of the $d$ spatial coordinates of) 
$\p'$ in the kernel of $\Omega$. 
From the structure of the kernel of $\Omega$,
this makes
$\p'$ an ``affine flex'' of $\p$, i.e. 
for all $i$, $\p'_i=A\p_i+\t$ for some matrix $A$
and vector $\t$.  

Suppose this affine flex is not trivial.
From~\cite[Lemmas A.3, A.7]{cgPss}, this implies
that $(\bar{G^*},\hp)$ has its ``edges on a conic at infinity''
(see that paper for definitions).
Since this is a coned framework, from~\cite[Lemma 4.10]{conic}
the framework must be ``ruled'' (see that paper for definitions). 
But the one skeleton of a convex polytope cannot be
ruled (\cite[Prop 3.4]{conic}.)

Thus for a non-trivial flex $\hp$, the inequality of 
Equation (\ref{eq:pss}) must hold, making $(G^*,\hp)$
prestress stable.

\end{proof}

\def\v{\oldv}
\bibliographystyle{abbrvlst}
\bibliography{framework}

\end{document}